\documentclass[11pt]{amsart}

\usepackage{amsmath,amsthm,amssymb,mathrsfs}

\usepackage[hidelinks]{hyperref}

\newtheorem{theorem}{Theorem}
\newtheorem{proposition}[theorem]{Proposition}
\newtheorem{corollary}[theorem]{Corollary}
\newtheorem{lemma}[theorem]{Lemma}

\theoremstyle{definition}

\newtheorem{remark}[theorem]{Remark}
\newtheorem{example}[theorem]{Example}

\newcommand{\R}{\ensuremath{\mathbb{R}}}
\newcommand{\s}{\ensuremath{\mathbb{S}}}

\newcommand{\diff}{\mathrm{d}}
\newcommand{\cH}{\ensuremath{\mathcal{H}}}
\newcommand{\barg}[1]{g\left( #1\right)}
\newcommand{\inner}[1]{\left\langle #1 \right\rangle}
\newcommand{\dd}[2]{\ensuremath{\left.\frac{\diff}{\diff #1}\right|_{#2}}}
\newcommand{\cR}{\ensuremath{\mathcal{R}}}

\newcommand{\II}{\mathop{}\!\mathrm{I\!I}}

\newcommand{\pz}{\ensuremath{{\mathop{}\!\partial_z}}}
\newcommand{\pbz}{\ensuremath{{\mathop{}\!\partial_{\bar{z}}}}}

\newcommand{\hess}{\operatorname{Hess}}
\newcommand{\Ric}{\operatorname{Ric}}

\newcommand{\tr}{\operatorname{tr}}	
\newcommand{\ii}{\ensuremath{\sqrt{-1}}}
\newcommand{\set}[1]{\left\{#1\right\}}

\begin{document}

\title[Hopf type theorem for surfaces in Riemannian manifolds]{Hopf type theorem for surfaces of constant weighted mean curvature in Riemannian manifolds}
\author{Jiaming Chen}
\address{School of Mathematics \\  Guangxi University \\ Nanning 530004, P. R. China}
\email{1067036321@qq.com}

\author{Shanze Gao}
\address{School of Mathematics \\  Guangxi University \\ Nanning 530004, P. R. China}
\address{Guangxi Center for Mathematical Research, Guangxi University, Nanning 530004, P. R. China}
\email{gaoshanze@gxu.edu.cn}


\begin{abstract}
We consider compact immersed surfaces of genus zero in a three-dimensional smooth metric measure space \((M^3,g,e^{-f}\mathrm{d}V_{g})\). We introduce a new weighted mean curvature and prove that any such surface whose weighted mean curvature is constant must be totally umbilical, provided that a condition relating the Ricci curvature of $M$ and the weight function $f$ is satisfied. 
We also give an application of this result to surfaces in $\mathbb{R}^3$ satisfying $|x|^{2} H + 2\langle x,\nu \rangle  = c$.
\end{abstract}

\maketitle

\section{Introduction}
\label{sec:introduction-3-mfd}

The classical Hopf theorem (see \cite{Hopf1951,hopf2013differential}) states that 
any constant mean curvature (CMC) immersion of a compact surface of genus $0$ into $\R^3$ is totally umbilical and hence a round sphere.
The assumption on genus is necessary since there exists a CMC torus immersed in $\R^3$ (see \cite{Wente1986}). On the other hand, for a compact immersed surface with genus $0$, the Hopf-type theorem holds provided that a curvature condition weaker than the CMC condition is satisfied. Interested readers may see \cite{hopf2013differential,HartmanWintner1954,Chern1955,AlencarCarmoTribuzy2007,KoisoPalmer2010} etc.

The proof of Hopf is based on the fact that the CMC condition is equivalent to the holomorphicity of a quadratic differential, which is often called the Hopf differential. Abresch and Rosenberg \cite{AbreschRosenberg2004} generalize the differential for 
CMC surfaces in the product spaces $\mathbb{S}^2\times\mathbb{R}$ and $\mathbb{H}^2\times\mathbb{R}$.
He and Li \cite{HeLi2009} introduce an anisotropic Hopf differential,  which is holomorphic if and only if a complex anisotropic mean curvature is constant.

In the present paper, we consider surfaces in a smooth metric measure space $(M^3,g,e^{-f}\diff V_{g})$, where $(M^3,g)$ is a $3$-dimensional Riemannian manifold, $\diff V_{g}$ is the volume form with respect to $g$, and $f$ is a real-valued smooth function on $M$. Let $x:\Sigma\to M$ be an immersed surface. The mean curvature and the unit normal vector of $\Sigma$ are denoted by $H$ and $\nu$ respectively. Using a weighted Hopf differential (see Section \ref{sec:Hopf-differential-3-mfd}), we show the following result.

\begin{theorem}
    \label{thm:Hopf-type-thm-3-mfd}
	Assume that there exists a function $\lambda$ on $M$ such that the Ricci curvature of $M$ satisfies
	\begin{equation}\label{eq:Ric}
        \Ric+ \frac{1}{2}\hess f + \frac{1}{4}\diff f\otimes\diff f=\lambda g .
    \end{equation}
    If $x:\Sigma\to M$ is a compact immersed surface of genus $0$ and its mean curvature satisfies
	\begin{equation}\label{eq:efH}
        e^{\frac{1}{2}f}\left( H +\frac12 \diff f(\nu) \right) = c
    \end{equation}
    on $\Sigma$ for some constant $c$, then $\Sigma$ is totally umbilical.
\end{theorem}

\begin{remark}
We recall the $m$-Bakry-\'Emery Ricci tensor (see, e.g.,  \cite{wei2009comparison})
\begin{equation*}
\Ric_f^m = \Ric+\hess f - \frac{1}{m}\diff f \otimes \diff f.
\end{equation*}
Then the condition \eqref{eq:Ric} can be written as
\begin{equation*}
\Ric_{\frac{f}{2}}^{-1}=\lambda g.
\end{equation*}
There are also some similarities between \eqref{eq:Ric} and the gradient Ricci soliton equation
\begin{equation*}
\Ric+\hess f=\lambda g.
\end{equation*}
Specially, when $f$ is constant, the condition \eqref{eq:Ric} means that $M$ is an Einstein manifold.
\end{remark}

We give some examples of spaces satisfying the condition \eqref{eq:Ric}.
\begin{example}\label{ex}
Let $M=\mathbb{R}^3 \setminus \{0\}$ with the standard metric $g=\sum_{i=1}^3 \diff x_i^2$. We choose $f=2\log |x|^2$. Since the Ricci curvature vanishes, $\diff f=\sum_i\frac{4x_i}{|x|^2}\diff x_i$ and
\[ \hess f=\sum_{i,j}\left( \frac{4\delta_{ij}}{|x|^2}-\frac{8x_ix_j}{|x|^4}\right) \diff x_i \diff x_j, \]
we have
\[ \Ric+ \frac{1}{2}\hess f + \frac{1}{4}\diff f\otimes\diff f= \frac{2}{|x|^2}g, \]
which verifies the condition \eqref{eq:Ric}.
\end{example}

\begin{example}\label{ex:conformal-Einstein-mfd}
Let $(M^3,g)$ be a conformally Einstein manifold, i.e., there exists a smooth function $u$ such that $\tilde{g}=e^{2u}g$ and $\Ric_{\tilde{g}} = \tilde{\lambda} \tilde{g}$, where $\Ric_{\tilde{g}}$ is the Ricci curvature with respect to the metric $\tilde{g}$ and $\tilde{\lambda}$ is a function on $M$. We choose $f=-2u$. Then
\begin{align*}
\Ric + \frac12\hess f + \frac{1}{4}\diff f\otimes\diff f = \left( \tilde{\lambda} e^{-f} - \frac12\Delta f + \frac14|\diff f|^2 \right)g,
\end{align*}
which satisfies the condition \eqref{eq:Ric}.
\end{example}

\begin{remark}
The weighted mean curvature $H_f$ for a surface $\Sigma$ in $(M^3,g,e^{-f}\diff V_{g})$ is usually defined by 
\[ H_f=H+\frac12\diff f(\nu), \]
which is from the first variation of the weighted volume $\int_{\Sigma}e^{-f}\diff V_\Sigma$ of $\Sigma$. We may define a new weighted mean curvature as
\begin{equation*}
  \cH_f := e^{\frac{1}{2}f}H_f.
\end{equation*}
Then the condition \eqref{eq:efH} means that $\cH_f$ is constant. When $f$ is constant, the condition \eqref{eq:efH} simply reduces to $H$ being constant.
\end{remark}

As a direct corollary of Theorem \ref{thm:Hopf-type-thm-3-mfd}, we have the following result.
\begin{corollary}
Assume that there exists a function $\lambda$ on $M$ such that the Ricci curvature of $M$ satisfies \eqref{eq:Ric}.
If $x:\Sigma\to M$ is a compact immersed $f$-minimal surface (namely $H_f=0$) of genus $0$, then $\Sigma$ is totally umbilical.
\end{corollary}

A prominent example of $f$-minimal surfaces is the self-shrinker of mean curvature flow (see e.g. \cite{Huisken1990,ColdingMinicozzi2012}). Brendle \cite{Brendle2016} proves that the round sphere is the only embedded, compact self-shrinker of genus $0$ in $\mathbb{R}^3$. The embeddedness assumption is necessary since there exist immersed self-shrinkers which are homeomorphic to the sphere but not the round one (see \cite{Drugan2015}). Alencar et al. \cite{AlencarSilvaNetoZhouHopf-type-theorem} establish a Hopf type theorem under an assumption to exclude these nonround examples. Note that self-shrinkers in $\mathbb{R}^3$ do not satisfy the condition \eqref{eq:Ric}.

Applying Theorem \ref{thm:Hopf-type-thm-3-mfd} to Example \ref{ex}, we obtain the following result.
\begin{corollary}\label{cor:ex}
	If $x:\Sigma\to \mathbb{R}^3$ is a compact immersed surface of genus $0$ and satisfies
	\begin{equation}\label{eq:ex}
        |x|^{2} H + 2\langle x,\nu \rangle  = c
    \end{equation}
    on $\Sigma$ for some nonzero constant $c$, then $\Sigma$ is a round sphere determined by
\begin{equation}\label{eq:|x|}
|x|^2-2\langle x,x_0 \rangle =c|x-x_0|,
\end{equation}
where $x_0\in \mathbb{R}^3$ is a fixed point.
\end{corollary}

\begin{remark}
It can be checked that the round sphere given by \eqref{eq:|x|} still satisfies \eqref{eq:ex} when $c=0$. However, Corollary \ref{cor:ex} fails when $c=0$. A nonround sphere example is shown in Section \ref{sec:cor-counterex}.
\end{remark}

\begin{remark}
The condition \eqref{eq:ex} can be compared to the dual Christoffel-Minkowski problem. In fact, surfaces satisfying
\[ K + \frac{2\langle x,\nu \rangle}{|x|^2}=0 \]
arise from the isotropic dual Minkowski problem, where $K$ is the Gauss curvature. Interested readers may see \cite{IvakiMilman2023,LiWan2024} etc. for some related results.
\end{remark}

The paper is organized as follows.
In Section \ref{sec:preliminaries-3-mfd}, we fix notation and recall basic facts about Riemannian manifolds and isothermal coordinates on immersed surfaces. 
In Section \ref{sec:Hopf-differential-3-mfd}, we investigate the weighted Hopf differential and derive a key formula about $\bar{\partial}$-derivative of this differential (see Proposition \ref{prop:Hopf-diff-derivative-3-mfd}). Following that, we complete the proof of Theorem \ref{thm:Hopf-type-thm-3-mfd}.
In Section \ref{sec:cor-counterex}, we give the proof of Corollary \ref{cor:ex} and provide an example of a non-round sphere for the $c=0$ case.

\section{Preliminaries}
\label{sec:preliminaries-3-mfd}

In this section, we recall some facts on surfaces in Riemannian manifolds.

Let $(M^3,g)$ be a three-dimensional Riemannian manifold and let $f$ be a smooth function on $M$. Let $\nabla$ denote the Levi-Civita connection on $(M,g)$.
The $(1,3)$-type and $(0,4)$-type Riemannian curvature tensors of $M$ are defined by
\begin{equation*}
    R(X_1,X_2)X_3 = \nabla_{X_1} \nabla_{X_2} X_3 - \nabla_{X_2} \nabla_{X_1} X_3 - \nabla_{[X_1,X_2]} X_3
\end{equation*}
and
\begin{equation*}
    R(X_1, X_2, X_3, X_4) = \barg{R(X_1, X_2)X_3,X_4},
\end{equation*}
for smooth vector fields $X_1,X_2,X_3$ and $X_4$ on $M$. 
The Ricci curvature $\Ric$ of $M$ is defined by
\begin{equation*}
\Ric(X_1,X_2)=\sum_{i=1}^{3}R(e_i,X_1,X_2,e_i),
\end{equation*}
where $\{e_1,e_2,e_3\}$ is an orthonormal frame on $M$.

Let $x:\Sigma\to M$ be an immersed surface and $g_\Sigma:=x^* g$ be the induced metric on $\Sigma$. From the isothermal coordinates theorem,  we know that, for a point $p\in\Sigma$, there exists a neighborhood of $p$ with isothermal coordinates $(u,v)$ such that
\begin{equation*}
    g_\Sigma = \rho^2 (\diff u^2 + \diff v^2),
\end{equation*}
where $\rho = \rho(u,v)$ is a smooth positive function.

Set $z = u + \ii v$. In this local complex coordinate system, we denote
\begin{equation*}
    x_z =\diff x\left( \frac{\partial}{\partial z} \right) \quad \text{ and } \quad
    x_{\bar{z}} =\diff x\left(  \frac{\partial}{\partial \bar{z}} \right),
\end{equation*}
where
\[ \frac{\partial}{\partial z}=\frac{1}{2}\left(  \frac{\partial}{\partial u} - \ii \frac{\partial}{\partial v}\right) \quad \text{ and } \quad \frac{\partial}{\partial \bar{z}}=\frac{1}{2}\left(  \frac{\partial}{\partial u} + \ii \frac{\partial}{\partial v}\right). \]
It can be checked that
\begin{equation*}
    g(x_z,x_{\bar z}) = \frac{\rho^2}{2},\quad
    g(x_z,x_z) = g(x_{\bar z},x_{\bar z})=0.
\end{equation*}

Let $\nu$ denote the unit normal vector of $\Sigma$. Clearly,
\[ g(x_z,\nu) = g(x_{\bar z},\nu)=0. \]
The second fundamental form $\II$ of $\Sigma$ is defined by
\begin{equation*}
    \II(X,Y) = g(\nabla_X Y,\nu)
\end{equation*}
for vector fields $X$ and $Y$ on $\Sigma$.
The mean curvature $H$ is given by
\[ H =\frac12 \tr_{g_\Sigma} \II=\frac{\II(x_z,x_{\bar z})}{\barg{x_z,x_{\bar z}}}. \]
Let $\phi = \II(x_z,x_z)$. The conjugate of $\phi$ is given by $\bar{\phi}=\II(x_{\bar z},x_{\bar z})$.
Then the Gauss formula is
\begin{equation}
    \label{eq:CC-covariant-derivative-H-phi}
    \left\{
    \begin{aligned}
        \nabla_{x_z} x_z &= \frac{2\rho_z}{\rho}  x_z + \phi\nu, \\[2pt]
        \nabla_{x_{\bar{z}}} x_z &= \nabla_{x_z} x_{\bar{z}} = \frac{\rho^2}{2} H\nu, \\[2pt]
        \nabla_{x_{\bar{z}}} x_{\bar{z}} &= \frac{2\rho_{\bar{z}}}{\rho} x_{\bar{z}} + \bar{\phi}\nu,
    \end{aligned}
    \right.
\end{equation}
and the Weingarten formula is
\begin{equation}
    \label{eq:CC-covariant-derivative-normal}
    \left\{
    \begin{aligned}
        \nabla_{x_z} \nu &=  -H x_z  -\frac{2}{\rho^2} \phi x_{\bar{z}},\\[2pt]
        \nabla_{x_{\bar{z}}} \nu &=  -\frac{2}{\rho^2} \bar{\phi} x_z  -H x_{\bar{z}}.
    \end{aligned}
    \right.
\end{equation}

The weighted volume $V_f(\Sigma)$ of $\Sigma$ is defined by
\[ V_f(\Sigma) = \int_\Sigma e^{-f}\diff\Sigma, \]
where $\diff\Sigma$ is the volume form of $g_\Sigma$. 
The first variation of the weighted volume gives
\begin{equation}\label{eq:1st-variation-weighted-volume}
  \dd{t}{t=0} V_f(\Sigma_t) = -2\int_\Sigma \left( H + \frac12\diff f(\nu) \right) \barg{T,\nu} \diff\Sigma_f,
\end{equation}
where $T$ is a compactly supported variational vector field on $\Sigma$ and $\diff\Sigma_f=e^{-f}\diff\Sigma$. 

Denote $H_f = H + \frac12\diff f(\nu)$ and define
\begin{equation*}
  \cH_f := e^{\frac{1}{2}f}H_f.
\end{equation*}

\section{Weighted Hopf Differential}
\label{sec:Hopf-differential-3-mfd}

The Hopf differential of $\Sigma$ is
\[ \Phi = \phi\diff z^2=\II(x_z,x_z)\diff z^2. \]
Recall the fact that $\Phi$ vanishes identically on $\Sigma$ if and only if $\Sigma$ is totally umbilical.

In the weighted setting, as in \cite{AlencarSilvaNetoZhouHopf-type-theorem}, we may define the weighted Hopf differential $\Phi_f$ on $\Sigma$ by
\begin{equation}
    \label{eq:Phi-f}
    \Phi_f := e^{-\frac{1}{2}f}\Phi.
\end{equation}
Let $\phi_f = e^{-\frac{1}{2}f}\phi$. Then $\Phi_f = \phi_f \diff z^2$ is clearly a quadratic differential.

\begin{proposition}\label{prop:phif=0}
	Assume that $\Sigma$ is compact. Then $\Sigma$ is totally umbilical if and only if $\Phi_f\equiv 0$.
\end{proposition}

\begin{proof}
    The compactness of $\Sigma$ guarantees that $e^{-\frac{1}{2}f}$ is positive and bounded, so $\Phi\equiv 0$ if and only if $\Phi_f\equiv 0$. The proof is completed by the fact that $\Sigma$ is totally umbilical if and only if $\Phi \equiv 0$.
\end{proof}

For convenience, let $\pz=\frac{\partial}{\partial z}$ and $\pbz=\frac{\partial}{\partial \bar{z}}$. We define a $(0,2)$-tensor $\cR_f$ on $M$ as
\begin{equation}
    \label{eq:R-f}
  \cR_f := \Ric + \frac{1}{2}\hess f + \frac{1}{4}\diff f\otimes\diff f.
\end{equation}
Now we derive the fundamental formula for the $\bar{\partial}$-derivative of the weighted Hopf differential $\Phi_f$.

\begin{proposition}
    \label{prop:Hopf-diff-derivative-3-mfd}
    In a local complex coordinate system on $\Sigma$, we have
    \begin{equation}
        \pbz\phi_f = e^{-\frac{1}{2}f}\barg{x_{\bar{z}},x_z}\left( e^{-\frac{1}{2}f}\pz\left( \cH_f \right) - \cR_f(x_z,\nu)\right).
    \end{equation}
\end{proposition}

\begin{proof}

Direct computation shows 
    \begin{align*}
        \partial_{\bar{z}}\phi &= \pbz \left( \barg{\nabla_{x_z}x_z,\nu}\right)  = \barg{\nabla_{x_{\bar z}}\nabla_{x_z}x_z,\nu} + \barg{\nabla_{x_z}x_z,\nabla_{x_{\bar z}}\nu}\\
        &= \barg{\nabla_{x_z}\nabla_{x_{\bar{z}}}x_z,\nu} + \barg{\nabla_{x_z}x_z,\nabla_{x_{\bar{z}}}\nu} + R(x_{\bar{z}},x_z,x_z,\nu),
    \end{align*}
where $[x_z,x_{\bar{z}}] = 0$ is used in the last equality.

    From \eqref{eq:CC-covariant-derivative-H-phi} and \eqref{eq:CC-covariant-derivative-normal}, we have 
    \begin{align*}
        \barg{\nabla_{x_z}\nabla_{x_{\bar{z}}}x_z,\nu} &= \barg{\nabla_{x_z}\left(\frac{\rho^2}{2}H\nu\right),\nu} \\
        &= \frac{\rho^2}{2}\partial_z H + \rho\rho_z H,
    \end{align*}
    and
	\begin{align*}
	\barg{\nabla_{x_z}x_z,\nabla_{x_{\bar{z}}}\nu}  = -\rho\rho_zH.
	\end{align*}	
    
	Thus, we have
    \begin{equation}\label{eq:pzbphi}
        \partial_{\bar{z}}\phi = \barg{x_z,x_{\bar z}}\partial_z H + R(x_{\bar{z}},x_z,x_z,\nu).
    \end{equation}

	 Now we claim that
    \begin{equation}\label{eq:R}
        R (x_{\bar{z}},x_z,x_z,\nu) = -g(x_{\bar{z}},x_z)\Ric(x_z,\nu).
    \end{equation}
    In fact, taking 
    $\set{e_1=\frac{x_u}{\rho},e_2=\frac{x_v}{\rho},e_3=\nu}$
    as a local orthonormal frame on $M$, we know
	\[ x_z=\frac{\rho}{2}(e_1 - \ii e_2) \quad \text{ and } \quad x_{\bar{z}}=\frac{\rho}{2}(e_1 + \ii e_2). \] 
	Then
	\begin{align*}
        R (x_{\bar{z}},x_z,x_z,\nu) 
        &=\frac{\rho^3}{8}R\left(e_1+\ii e_2,e_1-\ii e_2,e_1-\ii e_2,\nu\right)\\
        & = \frac{\rho^3}{4}\left(\ii R(e_2,e_1,e_1,\nu)+R(e_2,e_1,e_2,\nu)\right),  \\
		\Ric(x_z,\nu)&=\frac{\rho}{2}\left(\Ric\left(e_1-\ii e_2,\nu\right)\right) \\
		&=\frac{\rho}{2}\left(R(e_2,e_1,\nu,e_2)-\ii R(e_1,e_2,\nu,e_1)\right),
    \end{align*}
    which together yield \eqref{eq:R}.

	Combining \eqref{eq:pzbphi} and \eqref{eq:R}, we have
    \begin{align*}
        \partial_{\bar{z}}\phi = \barg{x_{\bar{z}},x_z}\left( \pz H - \Ric(x_z,\nu) \right).
    \end{align*}

    By a direct computation and the formula above, we have
    \begin{align*}
        \partial_{\bar{z}}\phi_f  &= e^{-\frac{1}{2}f}\left(\pbz\phi - \frac{1}{2}\phi\pbz f\right)\\
		&= e^{-\frac{1}{2}f} \barg{x_{\bar{z}},x_z} \left( \pz H - \Ric(x_z,\nu) - \frac{\phi\pbz f}{\rho^2} \right).
    \end{align*}

	On the other hand,
	\begin{align*}
	e^{-\frac{1}{2}f}\pz\left( \cH_f \right)&=\pz H_f  + \frac{1}{2}\pz f H_f \\
	&=\pz H +\frac12 \pz(\diff f(\nu))+\frac{1}{2}\pz f H+\frac{1}{4}\pz f \diff f(\nu).
	\end{align*}
	From \eqref{eq:CC-covariant-derivative-normal}, we have
    \begin{align*}
        \pz(\diff f(\nu)) 
        & = \hess f(x_z,\nu) - H\pz f - \frac{2}{\rho^2}\phi\pbz f.
    \end{align*}
	Thus
	\begin{align*}
	e^{-\frac{1}{2}f}\pz\left( \cH_f \right)
	&=\pz H + \left( \frac12\hess f + \frac{1}{4}\diff f\otimes\diff f \right)(x_z,\nu)- \frac{\phi\pbz f}{\rho^2}.
	\end{align*}
	
	Finally,
    \begin{equation*}
        \pbz\phi_f  = e^{-\frac{1}{2}f}\barg{x_{\bar{z}},x_z}\left( e^{-\frac{1}{2}f}\pz\left( \cH_f  \right) - \cR_f (x_z,\nu)\right).
    \end{equation*}
    
\end{proof}

\begin{proof}[Proof of Theorem \ref{thm:Hopf-type-thm-3-mfd}]

Under the assumptions of Theorem \ref{thm:Hopf-type-thm-3-mfd}, we know that \[ \cR _f(x_z,\nu) = \lambda\barg{x_z,\nu} = 0 \] and $\pz\cH_f  = 0$ hold all over $\Sigma$. By Proposition \ref{prop:Hopf-diff-derivative-3-mfd}, we obtain $\pbz\phi_f\equiv 0$ which indicates the weighted Hopf differential $\Phi_f$ is holomorphic. Since every holomorphic quadratic differential vanishes identically on a compact Riemann surface $\Sigma$ of genus $0$ (see \cite[Page 140]{hopf2013differential}), we know $\Phi_f\equiv 0$ on $\Sigma$. Then Proposition \ref{prop:phif=0} implies $\Sigma$ is totally umbilical.
\end{proof}

\section{On surfaces satisfying $|x|^{2} H + 2\langle x,\nu \rangle  = c$}
\label{sec:cor-counterex}

In this section, we study surfaces in $\mathbb{R}^3$ satisfying the condition
\[  |x|^{2} H + 2\langle x,\nu \rangle  = c.  \]
When $c\neq 0$, we can apply Theorem \ref{thm:Hopf-type-thm-3-mfd} as follows.

\begin{proof}[Proof of Corollary \ref{cor:ex}]
From \eqref{eq:ex} and $c\neq 0$, we know that the origin $0$ does not lie on the surface $\Sigma$. 
In other words, $\Sigma$ is a surface in $\mathbb{R}^3 \setminus \{0\}$. Then we can apply Theorem \ref{thm:Hopf-type-thm-3-mfd} to this case as shown in Example \ref{ex}. As a result, $\Sigma$ is totally umbilical, and hence a round sphere. The expression \eqref{eq:|x|} is obtained by direct computation.
\end{proof}

Now, we show that Corollary \ref{cor:ex} fails when $c=0$. Precisely, we will give an immersion $x:\mathbb{S}^2\rightarrow \mathbb{R}^3$ which satisfies
\begin{equation}\label{eq:c=0}
|x|^{2} H + 2\langle x,\nu \rangle  = 0,
\end{equation}
but is not the round sphere.

Define the inversion $I : \R^3 \setminus \{0\} \rightarrow \R^3$ by $I(p) = \frac{p}{|p|^2}$. Let $x:\Sigma \rightarrow \mathbb{R}^3\setminus \{0\}$ be an immersed surface. After the inversion, we obtain another immersed surface $\tilde{x}=I\circ x:\Sigma \rightarrow \mathbb{R}^3$. Let $H$ and $\tilde{H}$ denote the mean curvatures of $x$ and $\tilde{x}$, respectively. By calculation (see \cite{Willmore} for example), we know the following equality
\begin{equation*}
|x|^{2} H + 2\langle x,\nu \rangle = -\tilde{H} .
\end{equation*}
This relation inspires us to consider a complete minimal surface in $\mathbb{R}^3$ to construct the surface satisfying \eqref{eq:c=0}.

In \cite{Kusner1987}, Kusner constructs the following Weierstrass representative of minimal surfaces:
\begin{align*}
g_p(z) = \frac{z^{p-1}(z^p - s)}{(sz^p + 1)}, \qquad 
f_p(z) = \frac{\sqrt{-1}(sz^p + 1)^2}{(z^{2p} + rz^p - 1)^2},
\end{align*}
where $s=\sqrt{2p-1}$ and $r=\frac{2s}{p-1}$. 

By taking $p=2$, we obtain a minimal immersion $X=\operatorname{Re}F(z)$, where 
\[ F(z) = \frac{1}{D(z)}\left(\ii(z^3-z), z^3+z, \frac{\ii}{2}(z^4+1)+D(z) \right) \]
and $D(z)=z^4+2\sqrt{3}\,z^2-1$. As pointed out in \cite{Kusner1987}, this minimal surface is complete with four embedded flat ends and can be parametrized by
\[ X:\mathbb{S}^2\setminus\{\alpha_1,\alpha_2,\alpha_3,\alpha_4\}\to\mathbb{R}^3, \] where $\mathbb{S}^2$ is identified with the Riemann sphere $\mathbb{C}\cup \{\infty\} $, and $\alpha_1,\alpha_2,\alpha_3,\alpha_4$ are the four roots of $D$ in $\mathbb{C}$.

It can be checked that the minimal surface does not contain the origin $0$. Then we can define an immersion 
\begin{equation}\label{eq:x=IcX}
x= I\circ X:\mathbb{S}^2\setminus\{\alpha_1,\alpha_2,\alpha_3,\alpha_4\}\to\mathbb{R}^3.
\end{equation}

\begin{lemma}\label{lemma:extend}
    The immersion $x$ can be extended to $x:\s^2\to\mathbb{R}^3$ with $x(\alpha_i)=0$ for $ i=1,2,3,4$.
\end{lemma}

\begin{proof}
Note that the roots of $D$ are simple. We can check that, at each root $\alpha\in \{ \alpha_1,\alpha_2,\alpha_3,\alpha_4 \}$,
\[  F(z)=\frac{A}{z-\alpha}+B+O(z-\alpha), \]
where $A,B$ are complex vectors in $\mathbb{C}^3$ and $A\neq 0$.

From
\[ \langle F'(z),F'(z) \rangle= \langle X_u-\sqrt{-1}X_v,X_u-\sqrt{-1}X_v \rangle=0 \]
and
\[ F'(z)=-\frac{A}{(z-\alpha)^2}+O(1), \]
we know $\langle A,A \rangle =0$. Writing $A=a+\sqrt{-1} b$ with $a,b\in\R^3$, we have
\begin{equation}\label{eq:|a|}
|a|=|b|>0,\qquad \langle a,b \rangle=0.
\end{equation}

Then, letting $z=u+\sqrt{-1}v$,
\begin{align*}
X(u,v) = \operatorname{Re}F(z) =\frac{\operatorname{Re}(z-\alpha)a +  \operatorname{Im}(z-\alpha)b }{|z-\alpha|^2}+R(u,v),
\end{align*}
where $R$ is a smooth function near $\alpha$.

Denote $q=\operatorname{Re}(z-\alpha)a +  \operatorname{Im}(z-\alpha)b$. By \eqref{eq:|a|}, we have 
\[ |q|^2=|z-\alpha|^2|a|^2. \]
Clearly,
	\[ X=\frac{q}{|z-\alpha|^2}+R \quad \text{and}\quad |X|^2=\frac{|a|^2}{|z-\alpha|^2}+|R|^2+\frac{2\inner{ q,R} }{|z-\alpha|^2}.  \]
Furthermore,
\begin{align*}
x=\frac{X}{|X|^2}=\frac{q+|z-\alpha|^2R}{|a|^2+|z-\alpha|^2|R|^2+2\langle q,R \rangle}.
\end{align*}

Since $q\rightarrow 0$ as $z \rightarrow \alpha$ and $|a|>0$, we know $x(u,v)\rightarrow 0$ as $u+\sqrt{-1}v\rightarrow \alpha$. Hence, we set $x(\operatorname{Re}\alpha,\operatorname{Im}\alpha)=0$.

Finally, we show that the differential $dx$ is injective at $(\operatorname{Re}\alpha,\operatorname{Im}\alpha)$. Actually, by calculation at $(\operatorname{Re}\alpha,\operatorname{Im}\alpha)$,
 \[ \diff x \left(\frac{\partial}{\partial u}\right)=\frac{a}{|a|^2}, \qquad  \diff x\left(\frac{\partial}{\partial v}\right)=\frac{b}{|a|^2}.
 \]
The proof is completed by noticing that $a$ and $b$ are linearly independent from \eqref{eq:|a|}.
\end{proof}

We finish this section with the following theorem.
\begin{theorem}
There exists a non-round immersed sphere $x:\mathbb{S}^2\rightarrow \mathbb{R}^3$ satisfying \eqref{eq:c=0}.
\end{theorem}

\begin{proof}
By Lemma \ref{lemma:extend}, the immersion $x$ defined by \eqref{eq:x=IcX} can be extended to an immersion on $\mathbb{S}^2$ with $x(\alpha_1)=\dots=x(\alpha_4)=0$.
Since $X$ is minimal and $x$ vanishes at these four points,
we know that $x$ satisfies \eqref{eq:c=0}. It is not the standard round sphere because $x$ is not an embedding.
\end{proof}

\section*{Acknowledgments}
The authors would like to express their gratitude to Professor Haizhong Li and Professor Hui Ma for their valuable comments.

\bibliographystyle{amsplain}

\end{document}